\documentclass[12pt,reqno]{amsart}
\usepackage{amsmath,amsthm,amssymb,amsfonts,amscd}
\usepackage{mathrsfs}
\usepackage{bbding}
\usepackage{graphicx,latexsym}
\usepackage{hyperref}
\usepackage{geometry}
\numberwithin{equation}{section}

\theoremstyle{plain}
\newtheorem{theorem}{Theorem}[section]
\newtheorem*{theorem*}{Theorem}
\newtheorem{lemma}[theorem]{Lemma}

\newtheorem{proposition}[theorem]{Proposition}

\theoremstyle{definition}
\newtheorem{definition}{Definition}[section]

\theoremstyle{remark}
\newtheorem{remark}{Remark}

\renewcommand{\Re}{\operatorname{Re}}
\renewcommand{\Im}{\operatorname{Im}}

\newcommand{\supp}{\operatorname{supp}}
\newcommand{\bs}{\backslash}
\renewcommand{\le}{\left}
\newcommand{\ri}{\right}
\newcommand{\la}{\langle}
\newcommand{\ra}{\rangle}
\renewcommand{\mod}{\operatorname{mod}\ }

\newcommand{\cA}{\mathcal{A}}

\newcommand{\cC}{\mathcal{C}}

\newcommand{\cE}{\mathcal{E}}
\newcommand{\cF}{\mathcal{F}}

\newcommand{\cL}{\mathcal{L}}
\newcommand{\cM}{\mathcal{M}}

\newcommand{\cP}{\mathcal{P}}

\newcommand{\fa}{\mathfrak{a}}
\newcommand{\fb}{\mathfrak{b}}

\newcommand*{\bbC}{\ensuremath{\mathbb{C}}}

\newcommand*{\bbH}{\ensuremath{\mathbb{H}}}

\newcommand*{\bbR}{\ensuremath{\mathbb{R}}}

\newcommand*{\bbZ}{\ensuremath{\mathbb{Z}}}

\renewcommand{\d}{\delta}
\newcommand{\D}{\Delta}
\newcommand{\f}{\frac}
\newcommand{\g}{\gamma}
\newcommand{\G}{\Gamma}
\renewcommand{\l}{\lambda}

\newcommand{\s}{\sigma}

\newcommand{\ve}{\varepsilon}
\newcommand{\vp}{\varphi}

\begin{document}

\title[Sup-norm bounds for Eisenstein series]{Sup-norm bounds for Eisenstein series}
\author{Bingrong Huang and Zhao Xu}
\address{School of Mathematics \\ Shandong University \\ Jinan \\Shandong 250100 \\China}
\email{brhuang@mail.sdu.edu.cn}
\email{zxu@sdu.edu.cn}
\date{\today}

\begin{abstract}
  The paper deals with establishing bounds for Eisenstein series
  on congruence quotients of the upper half plane, with control
  of both the spectral parameter and the level. The key observation
  in this work is that we exploit better the structure of the
  amplifier by just supporting on primes for the Eisenstein series,
  which can use both the analytic method as Young did to get a
  lower bound for the amplifier and the geometric method as
  Harcos--Templier did to obtain a more efficient treatment
  for the counting problem.
\end{abstract}

\thanks{The first author is supported in part by the project of the National Natural Science Foundation of China (11531008), and the second author is supported by the project of the National Natural Science Foundation of China
(11501327). }
\maketitle
\setcounter{tocdepth}{1}
\tableofcontents

\section{Introduction} \label{sec: Introduction}

Bounding the sup-norm of Laplace eigenfunctions on manifolds is a classical problem.
We shall establish new bounds for the well-studied modular surface $Y_0(q)=\G_0(q)\bs\bbH$
with its hyperbolic measure.

For Hecke--Maass cuspidal newforms $u_j$ of spectral parameter $t_j$ with $L^2$-normalized,
in the spectral aspect, the nontrivial bound is due to
Iwaniec--Sarnak \cite{iwaniec1995norms} (for level $q=1$)
and Blomer--Holowinsky \cite{blomer2010bounding} (for square-free level $q$) who established
\begin{equation}
  \|u_j\|_\infty \ll_{q,\ve} t_j^{5/12+\ve}
\end{equation}
for any $\ve>0$.
In the level aspect, the first non-trivial bound is due to Blomer--Holowinsky \cite{blomer2010bounding}
who proved that
\begin{align}
\|u_j\|_\infty \ll_{t_j} q^{-1/37},
\end{align}
for square-free $q$.
Then this is improved by Templier \cite{templier2010sup}, Harcos--Templier \cite{harcos2012sup,harcos2013sup},
and the current best result is
\begin{equation}
  \|u_j\|_\infty \ll_{t_j,\ve} q^{-1/6+\ve}.
\end{equation}
Hybrid bounds save a power simultaneously in the spectral and level aspects.
The following hybrid bound is established by Blomer--Holowinsky \cite{blomer2010bounding}
\begin{equation}
  \|u_j\|_\infty \ll t_j^{1/2} (t_j q)^{-1/2300},
\end{equation}
for square-free level $q$. In \cite{templier2015hybrid}, Templier obtains the following hybrid bound which
generalizes the best known bounds in the spectral and level aspects simultaneously
\begin{equation}\label{Templier}
  \|u_j\|_\infty \ll t_j^{5/12+\ve} q^{-1/6+\ve},
\end{equation}
for square-free level $q$.

It is natural to consider this problem for the continuous spectrum situation.
It seems to have been neglected until Young \cite[theorem 1.1]{young2015note}
establishes that for $\Omega$ a fixed compact subset of $\bbH$, and $T\geq 1$,
\begin{equation}\label{Young}
  \max_{z\in\Omega} |E(z,1/2+iT)| \ll_{\Omega,\ve} T^{3/8+\ve},
\end{equation}
where $E(z,s)$ is the usual real-analytic Eisenstein series for the group $PSL_2(\bbZ)$.
The Eisenstein series case is similar in some ways to the cuspidal case, but has some technical problems because
of the constant term in the Fourier expansion.
Let
\begin{equation}\label{eqn: F}
  F(z,s) = E(z,s) - y^{s} - \vp(s)y^{1-s},
\end{equation}
where
\begin{align}\label{eqn: vp(s)}
  \vp(s)=\xi(2(1-s))/\xi(2s),
\end{align}
with
\begin{align}\label{eqn: xi(s)}
  \xi(s)=\pi^{-s/2}\G(s/2)\zeta(s).
\end{align}
Although not stated explicitly in his main theorem, in \cite[Section 6]{young2015note},
Young actually derived that, for $z=x+iy\in\mathbb{H}$,
\[
  F(z,1/2+iT) \ll
  \left\{\begin{array}{ll}
    T^{3/8+\ve}, & \textrm{if } 1\ll y\ll T^{1/8}\  \textup{or}\  y\gg T^{1/4}, \\
    y^{1/3}T^{1/3+\ve}, & \textrm{if } T^{1/8}\ll y \ll T^{1/5}, \\
    (T/y)^{1/2}\log^2 T, & \textrm{if } T^{1/5}\ll y \ll T^{1/4}.
  \end{array}\right.
\]
The main impetus in Young's result is
the realization that one can choose an efficient amplifier for the Eisenstein series, which leads
to the improved exponent compared to the cusp form case (see \cite[Remark 1.6]{iwaniec1995norms}).

In this paper, we will give the bounds for the Eisenstein series
in a wider range of $y$, and then extend to bound the Eisenstein
series with square-free levels.
Our first main theorem is a generalization of Young's result.
\begin{theorem}\label{thm: level 1}
  Let $T\geq1$. For any $z\in \bbH$, we have
  \begin{equation}\label{eqn: E bound}
    E(z,1/2+iT) = y^{1/2+iT}+\varphi(1/2+iT)y^{1/2-iT} + O(y^{-1/2}+T^{3/8+\ve}).
  \end{equation}
\end{theorem}

\begin{remark}
  This result should be compared with Young's result \eqref{Young}.
  We actually get
  \begin{align}\label{remark for Thm1}
    E(z,1/2+iT)\ll T^{3/8+\ve},
  \end{align}
  provided that $T^{-3/4-\ve}\ll y\ll T^{3/4+\ve}$, which is a wider range.

\end{remark}

Now we turn to consider the Eisenstein series of level $q$,
where $q$ is a positive square-free integer.
Let $E_\fa(z,s)$ be the Eisenstein series for the congruence group $\G_0(q)$ (see \S \ref{subsec: ES}),
where $\fa$ is a cusp of $\Gamma_0(q)$. Our main result in this case is the following theorem.
\begin{theorem}\label{thm: level q}
  Let $z=x+iy\in\bbH$, $q$ be a positive square-free integer.
  Then for $T\geq 1$, we have
  \begin{equation}\label{eqn: E_a<<}
    E_\fa(z,1/2+iT) = \delta_\fa y^{1/2+iT} + \vp_{\fa}(1/2+iT)y^{1/2-iT} + O(q^{-1/2+\ve}(y^{-1/2}+T^{3/8+\ve})).
  \end{equation}
  In particular, if $\fa\sim\infty$, for all $y\gg1/q$, we have
  \begin{equation}\label{eqn: E_a<<p}
    E_\fa(z,1/2+iT) = y^{1/2+iT} + \vp_{\fa}(1/2+iT)y^{1/2-iT} + O(q^{-1/2+\ve}T^{3/8+\ve}).
  \end{equation}
\end{theorem}

As in Iwaniec--Sarnak's work \cite{iwaniec1995norms}, we shall use the amplification method which
relates the sup-norm problem to an interesting lattice point counting.
The key observation in this work is that we exploit better the structure of the amplifier
by just supporting on primes for the Eisenstein series,
which can use both the analytic method as Young \cite{young2015note} did to get a lower bound for the amplifier
and the geometric method as Harcos--Templier \cite{harcos2013sup} and Templier \cite{templier2015hybrid} did to
obtain a more efficient treatment for the counting problem.

\begin{remark}
  We may compare the error term in \eqref{eqn: E_a<<p} with the current result of
  Hecke--Maass cusp forms in \eqref{Templier}. we can see that improved exponents both in spectral aspect and
  level aspect. This is done because of the choice of the amplifier, and the fact that all Eisenstein series of square-free level are oldforms coming from the full level Eisenstein series.
  However, by our construction of amplifier, one may expect that
  \begin{equation}
    \|u_j\|_\infty \ll t_j^{3/8+\ve} q^{-1/4+\ve},
  \end{equation}
  if we assume
  \begin{equation}\label{eqn: assumption}
    \sum_{p\sim N} |\lambda_j(p)|^2 \gg_\ve (qt_j)^{-\ve} N^{1-\ve},
  \end{equation}
  as \cite[Remark 1.6]{iwaniec1995norms} did.
\end{remark}

\begin{remark}
  For Hecke--Maass cuspidal newforms, one can restrict the sup-norm problem to the region $z\in \cF(q)$ (see Definition \ref{def: F(q)}) by
  the Atkin--Lehner theory which tells us that each Hecke--Maass cuspidal newform is an eigenfunction
  of the corresponding Atkin--Lehner operators.
  When it comes to Eisenstein series, we can replace the Atkin--Lehner theory for Hecke--Maass cuspidal newforms by
  explicit computations on Eisenstein series (see \eqref{eqn: A_q E_fa} and \eqref{eqn: A_q E_a}).
  So it is natural to consider the case $y\gg1/q$ (see \S\ref{subsec: AL}).
\end{remark}

We state that, throughout the paper, $\ve$ is an arbitrarily small positive number
which may not be the same in each occurrence.
The paper is organized as follows.
Section \ref{sec: Preliminaries} is devoted to the background on Eisenstein series,
establishment of the first reduction of Theorem \ref{thm: level 1},
and the construction of our amplifier.
Bounds via Fourier expansion and a pointwise bound for Eisenstein series
via its integral are also summarized in Section \ref{sec: Preliminaries},
which is proved by Young \cite[Section 4]{young2015note}.
In Section \ref{sec: lower bound for amplifier},
a lower bound for our amplifier is established.
And then after a summary of the results on counting lattice points
from Templier \cite{templier2015hybrid} in Section \ref{sec: counting},
we complete the proof of Theorem \ref{thm: level 1} in Section \ref{sec: thm1}.
In the last section, we give the proof of Theorem \ref{thm: level q}.

\section{Preliminaries} \label{sec: Preliminaries}

In this section, we collect those basic facts which are needed in this paper.
We restrict $q$ to be a positive square-free integer.
Let $\mathbb{H}$ be the upper half-plane, $\G=SL_2(\bbZ)$ the full modular group,
and $\G_0(q)$ the Hecke congruence group of level $q$.

Let $\cA(\G\bs\bbH)$ denote the space of automouphic functions of weight zero, i.e.,
the functions $f:\bbH\rightarrow\bbC$ which are $\G$-periodic.
Let $\cL(\G\bs\bbH)$ denote the subspace of square-integrable functions with respect to the inner product
\begin{equation}\label{eqn: inner product}
  \la f,g \ra = \int_{\G\bs\bbH} f(z)\overline{g(z)}d\mu z,
\end{equation}
where $d\mu z = y^{-2}dxdy$ is the invariant measure on $\bbH$.
The Laplace operator
\begin{equation}
  \D = -y^2\left(\frac{\partial^2}{\partial x^2} + \frac{\partial^2}{\partial y^2} \right)
\end{equation}
acts in the dense subspace of smooth functions in $\cL(\G\bs\bbH)$ such that $f$ and
$\D f$ are both bounded; it has a self-adjoint extension which yields the spectral
decomposition $\cL(\G\bs\bbH) = \bbC\oplus\cC(\G\bs\bbH)\oplus\cE(\G\bs\bbH)$. Here $\bbC$ is the space of constant
functions, $\cC(\G\bs\bbH)$ is the space of cusp forms and $\cE(\G\bs\bbH)$ is the space of
Eisenstein series. We will focus on $\cE(\G\bs\bbH)$.

\subsection{Eisenstein series for $\Gamma$ and $\Gamma_0(q)$} \label{subsec: ES}

The Eisenstein series $E(z,s)$ for $\Gamma$ is defined as
\begin{equation} \label{eqn: ES}
  E(z,s) = \sum_{\g\in\G_\infty\bs\G} (\Im \gamma z)^s,
\end{equation}
if $\Re(s)>1$, and then by analytic continuation for all $s\in\bbC$.
Here $\Gamma_\infty$ is the stability group of $\infty$.
For any integer $\ell\geq1$, the Hecke operator $T_\ell$ is defined by
\begin{equation} \label{eqn: HO}
  (T_\ell f)(z) = \frac{1}{\sqrt{\ell}}\sum_{ad=\ell}\sum_{b(\mod d)} f\left(\frac{az+b}{d}\right).
\end{equation}
Moreover, the reflection operator $R$ defined by $(Rf)(z)=f(-\bar{z})$.
$f$ is called even or odd according to $(Rf)(z)=f(z)$ or $(Rf)(z)=-f(z)$.
All the Eisenstein series $E(z,s)$ are even and they are eigenfunctions of the Hecke operators
\begin{equation} \label{eqn: ES as an eigenfunc of HO}
  T_\ell E(z,s) = \eta(\ell,s)E(z,s),
\end{equation}
where
\begin{equation} \label{eqn: eta}
  \eta(\ell,s) = \sum_{ad=\ell}(a/d)^{s-1/2}.
\end{equation}
We will write $\eta_{it}(\ell)=\eta(\ell,1/2+it)$.
And for any $m,n\in\bbZ_{>0}$, they enjoy the Hecke relation:
\begin{equation}\label{eqn: HR eta}
\eta_{it}(m)\eta_{it}(n)=\sum_{d|(m,n)}\eta_{it}\left(\f{mn}{d^2}\right).
\end{equation}
We have the Fourier expansion
\begin{equation}\label{eqn: FE}
  E(z,s) = y^{s} + \varphi(s)y^{1-s} + \frac{2\sqrt{y}}{\xi(2s)}\sum_{n\neq0}\eta(|n|,s-1/2)K_{s-1/2}(2\pi |n|y)e(nx).
\end{equation}

Now we consider the Eisenstein series for $\Gamma_0(q)$,
and we collect several statements about them from \cite{conrey2000cubic}.
The Eisenstein series $E_\fa(z,s)$ for $\G_0(q)$ at a cusp $\fa$ is defined as
\begin{equation} \label{eqn: E_a}
  E_\fa(z,s) = \sum_{\g\in\G_\fa\bs\G_0(q)} (\Im \s_\fa^{-1}\g z)^s,
\end{equation}
if $\Re(s)>1$ and by analytic continuation for all $s\in\bbC$.
Here $\G_\fa$ is the stability group of $\fa$ and $\s_\fa\in SL_2(\bbR)$ such that
$\s_\fa\infty=\fa$ and $\s_\fa^{-1}\G_\fa\s_\fa=\G_\infty$.
The scaling matrix $\s_\fa$ is only determined up to a translation from the right;
however the Eisenstein series does not depend on the choice of $\s_\fa$, not even on the choice of a cusp in the equivalence class.
We give some explicit computations on Eisenstein series.
We use the details on \cite{duke2002subconvexity}.
Recall that $q$ is square-free, so every cusp of $\G_0(q)$ is equivalent to $\fa=1/v$ with $v|q$.
The complementary divisor $w=q/v$ is the width of $\fa$.
(In the case $w=1$, so $v=q$, which means $\fa\sim \infty$.)
Let $\bar{\fa}$ stand for the cusp ``dual'' to $\fa$ in the sense that $\bar{\fa}\sim1/w$ if $\fa\sim1/v$.
The scaling matrix of $\fa$ can be chosen as
\begin{equation}\label{eqn: s_fa}
  \s_\fa = \begin{pmatrix} \sqrt{w}&0\\ v\sqrt{w}&1/\sqrt{w} \end{pmatrix}.
\end{equation}
Thus, we have
\begin{equation}\label{eqn: s_fa^-1 Gamma}
  \s_\fa^{-1}\G_0(q) = \left\{ \begin{pmatrix} a/\sqrt{w}&b/\sqrt{w}\\ c\sqrt{w}&d\sqrt{w} \end{pmatrix}: \begin{pmatrix} a&b\\ c&d \end{pmatrix}\in SL_2(\bbZ), c+av\equiv 0(vw) \right\}.
\end{equation}
The coset $\G_\infty\bs\s_\fa^{-1}\G_0(q)$ is parametrized by pairs of integers $(c,d)=1$, $c\equiv0\pmod{v}$ and $(c/v,w)=1$.
Therefore the Eisenstein series for the cusp $\fa\sim1/v$ is given by
\begin{equation}\label{eqn: explicit formula}
  E_\fa(z,s) = \frac{1}{2}\le(\frac{y}{w}\ri)^s \overset{*}{\sum_c\sum_d} |cvz+d|^{-2s},
\end{equation}
where $*$ means the summation is over $(c,d)\in\bbZ^2$ with $(c,d)=1,\ (c,w)=1$ and $(v,d)=1$.
Then, by trivial computation, we have
\begin{equation}\label{eqn: A_q E_fa}
  \begin{split}
     E_\fa\left(\frac{-1}{qz},s\right) & = \frac{1}{2}\le(\frac{y}{q|z|^2 w}\ri)^s \overset{*}{\sum_c\sum_d} \left|-\f{cv}{qz}+d\right|^{-2s} \\
       & = \frac{1}{2}\le(\frac{y}{v}\ri)^s \overset{*}{\sum_c\sum_d} |dwz-c|^{-2s} \\
       & = E_{\bar{\fa}}(z,s).
  \end{split}
\end{equation}
And by \cite[Eq. (3.25)]{conrey2000cubic}, we have the following explicit relation between $E_\fa(z,s)$ and $E(z,s)$,
\begin{equation}\label{eqn: E_a to E}
  E_\fa(z,s) = \zeta_q(2s)\mu(v)(qv)^{-s}
     \sum_{\beta|v}\sum_{\gamma|w}\mu(\beta\gamma)\beta^s\gamma^{-s}E(\beta\gamma z,s),
\end{equation}
where
\begin{equation}\label{eqn: zeta_q}
  \zeta_q(s) = \prod_{p|q}(1-p^{-s})^{-1}.
\end{equation}

The Fourier expansion of the Eisenstein series at cusp $\fa$ is known (\cite[Theorem 3.4]{iwaniec2002spectral})
\begin{equation}\label{eqn: FE of E_a}
  E_\fa(z,s) = \d_{\fa}y^{s} + \vp_{\fa}(s)y^{1-s} + \sum_{n\neq0}\vp_{\fa}(n,s)W_s(nz),
\end{equation}
where $\d_\fa=1$ if $\fa\sim\infty$ or $\d_\fa=0$ otherwise, and
\begin{equation}\label{eqn: vp_a(s)}
  \vp_\fa(s) = \pi^{1/2}\frac{\G(s-1/2)}{\G(s)} \sum_c \frac{S_\fa(0,0;c)}{c^{2s}},
\end{equation}
\begin{equation}\label{eqn: vp_a(n,s)}
  \vp_\fa(n,s) = \pi^s \G(s)^{-1} |n|^{s-1} \sum_c \frac{S_\fa(0,n;c)}{c^{2s}},
\end{equation}
with $S_\fa(0,n;c)=S_{\fa\infty}(0,n;c)$ and
\begin{equation} \label{eqn: Kloosterman sums}
  S_{\fa\fb}(m,n;c) := \sum_{\left(\begin{smallmatrix}a&*\\c&d\end{smallmatrix}\right)\in B\bs\s_\fa^{-1}\G\s_\fb/B} e\left(\frac{md+na}{c}\right),
\end{equation}
\begin{equation} \label{eqn: B}
  B = \left\{ \begin{pmatrix}1&b\\0&1\end{pmatrix}: b\in\bbZ \right\},
\end{equation}
and $W_s(z)$ is the Whittaker function given by
\begin{equation}\label{eqn: Whittaker function}
  W_s(nz) = 2 \sqrt{|n|y} K_{s-1/2}(2\pi|n|y)e(nx).
\end{equation}

Now, let us end this subsection with a very rough introduction on Hecke--Maass cuspidal forms for $\Gamma$.
Let $\{u_j\}$ be an orthonormal basis of the space of Maass cusp forms
for $\Gamma$ such that
$\Delta u_j(z)=\lambda_ju_j(z)$
and
$T_n u_j(z)=\lambda_j(n)u_j(z)$.
The Hecke eigenvalues $\lambda_j(n)$ also enjoy the Hecke relation:
\begin{equation}\label{eqn: HR lambda_j}
  \lambda_j(m)\lambda_j(n)
  = \sum_{d|(m,n)} \lambda_j\left(\frac{mn}{d^2}\right).
\end{equation}


\subsection{Atkin--Lehner operators and a gap principle} \label{subsec: AL}

We let $GL_2(\bbR)^+$ act on the upper-half plane $\bbH$ by the usual fractional linear transformations.
If $q$ is square-free, for each divisor $d|q$,
we consider the matrices $W_d\in M_2(\bbZ)$ of determinant $d$ such that
\begin{equation}
  W_d\equiv \begin{pmatrix} *&*\\ 0&* \end{pmatrix}\pmod{q} \quad {\rm and} \quad W_d\equiv \begin{pmatrix} 0&*\\0&0 \end{pmatrix} \pmod{d}.
\end{equation}
Scaling the $W_d$'s by $1/\sqrt{d}$ we obtain matrices in $SL_2(\bbR)$ which is called the \emph{Atkin--Lehner operators}.
The Atkin--Lehner operators together for all $d|q$ form a subgroup $A_0(q)$ of $SL_2(\bbR)$
containing $\Gamma_0(q)$ as a normal subgroup. The quotient group $A_0(q)/\G_0(q)$ is
isomorphic to $(\bbZ/2\bbZ)^{\omega(q)}$, where $\omega(q)$ is the number of distinct prime factors of $q$.

\begin{definition}\label{def: F(q)}
  Let $\cF(q)$ be the set of $z\in\bbH$ such that $\Im(z)\geq\Im(Az)$ for all $A\in A_0(q)$,
  which is a fundamental domain for $A_0(q)$.
\end{definition}

For $q$ being prime, it is easy to see that
$A_0(q) = \G_0(q)\cup A_q\G_0(q)$, where
\begin{equation}
  A_q=\frac{1}{\sqrt{q}}\begin{pmatrix} &-1\\q& \end{pmatrix}.
\end{equation}
By the definition of $E_\fa(z,s)$, we have
$E_\fa(\g z,s)=E_\fa(z,s)$,
if $\g\in\G_0(q)$. And by \eqref{eqn: A_q E_fa}, we have
\begin{equation} \label{eqn: A_q E_a}
  A_q E_\fa(z,s) = E_\fa\left(\frac{-1}{qz},s\right) = E_{\bar{\fa}}(z,s).
\end{equation}
Hence for $A\in A_0(q)$, we have $E_\fa(Az,s)=E_\fb(z,s)$, where $\fb=\fa$ if $A\in\G_0(q)$,
and $\fb=\bar{\fa}$ if $A\in A_q\G_0(q)$.

For Hecke--Maass cuspidal newform $u_j$, by Atkin--Lehner theory we know that it is an eigenvector
for the Atkin--Lehner operators with eigenvalues $\pm1$.
Therefore, we may assume that $z\in\cF(q)$ when investigating the sup-norm of a Hecke--Maass cuspidal newform.
By \eqref{eqn: A_q E_a}, we can still make the same assumption for Eisenstein series,
that is, $z\in\cF(q)$. Note that for $z\in \cF(q)$, by \cite[Lemma 2.2]{harcos2012sup}, we have
\begin{equation}
    \Im z \geq \frac{\sqrt{3}}{2q},
\end{equation}
and
$$
|cz+d|^2\geq 1/q,
$$
where $(c,d)\in \mathbb{Z}^2$ is distinct from $(0,0)$.


\subsection{Amplified pre-trace formula} \label{subsec: APTF}

Let $k\in\mathcal{C}^\infty([0,\infty))$ with rapid decay.
Then, it can be viewed as the inverse of the Selberg transform
of a function $h(t)$, which is given by the following three steps
(see \cite[(1.64)]{iwaniec2002spectral}:
\begin{equation*}
  \begin{split}
    g(\xi)&:=\f{1}{2\pi i}\int_{-\infty}^\infty e^{-ir\xi}h(r)dr, \\
    2q(v) &:=g(2\log(\sqrt{v+1}+\sqrt{v})),\\
    k(u)  &:=-\f{1}{\pi}\int_{u/4}^{\infty}(v-u/4)^{-1/2}dq(v).
  \end{split}
\end{equation*}
Assume now that $k(z,w)=k(u(z,w))$ is a point point-pair
invariant kernel
with
\begin{equation*}
  u(z,w):=\f{|z-w|^2}{\Im (z)\Im(w)},\ \ z,w\in\mathbb{H},
\end{equation*}
and $h(t)$ is the corresponding Selberg transform which satisfies the conditions (see \cite[(1.63)]{iwaniec2002spectral})
\begin{equation}\label{eqn: h(r) conditions}
  \left\{\begin{array}{l}
    h(t)  \textrm{ is even,}  \\
    h(t)  \textrm{ is holomorphic in the strip } |\Im t | < \frac{1}{2} + \varepsilon, \\
    h(t)  \ll (|t|+1)^{-2-\varepsilon}  \textrm{ in the strip.}
  \end{array}\right.
\end{equation}
Due to \cite[Theorem 7.4]{iwaniec2002spectral}, we have the spectral expansion
\begin{equation*}
  \begin{split}
    K(z,w): & = \sum_{\g\in\Gamma} k(\gamma z,w) \\
            & = \sum_{j\geq0}h(t_j)u_j(z)\overline{u_j(w)}
                + \frac{1}{4\pi} \int_{-\infty}^{\infty} h(r) E(z,1/2+ir)\overline{E(w,1/2+ir)}dr.
  \end{split}
\end{equation*}
Then it follows by applying the $\ell$th Hecke operator on the spectral expansion
that (cf. displays (1.3) and (1.4) in \cite{iwaniec1995norms})
\begin{equation} \label{eqn: pretrace formula}
  \begin{split}
    &\frac{1}{\sqrt{\ell}} \sum_{\g\in \cM(\ell)} k(\g z,z) \\
    =&\sum_{j\geq0}\l_j(\ell)h(t_j)|u_j(z)|^2 + \frac{1}{4\pi} \int_{-\infty}^{\infty} \eta_{ir}(\ell) h(r) |E(z,1/2+ir)|^2 dr,
  \end{split}
\end{equation}
where $\cM(\ell)$ is the set of matrices
$\g=\left(\begin{smallmatrix}a&b\\c&d\end{smallmatrix}\right) \in M_2(\bbZ)$
with $\det(\g)=\ell$.

Now we consider the following sum
\[
  \begin{split}
    \sum_{j\geq0}h(t_j)A_j |u_j(z)|^2 + \frac{1}{4\pi} \int_{-\infty}^{\infty}  h(r) A_{ir} |E(z,1/2+ir)|^2 dr,
  \end{split}
\]
where
\begin{equation}\label{eqn: A}
  A_j = \bigg|\sum_{n}x_n\l_j(n)\bigg|^2,\quad
  A_{ir} = \bigg|\sum_{n}x_n\eta_{ir}(n)\bigg|^2,
\end{equation}
and $h:\bbR\cup[-i/2,i/2]\rightarrow \bbR_+$ is a positive even smooth function
of rapid decay, $(x_n)$ is sequence of complex numbers supported on finitely many $n$'s.
After squaring out the $n$-sum and applying
the Hecke relations \eqref{eqn: HR eta} and \eqref{eqn: HR lambda_j}, we arrive at
\[
  \sum_{\ell}|y_\ell| \left|\sum_{j\geq0}\l_j(\ell)h(t_j)|u_j(z)|^2
  + \frac{1}{4\pi} \int_{-\infty}^{\infty} \eta_{ir}(\ell) h(r) |E(z,1/2+ir)|^2 dr\right|,
\]
where
\begin{equation}\label{eqn: y_ell}
  y_\ell := \sum_{\substack{d|(m,n)\\ \ell=mn/d^2}} x_m\overline{x_n}
  = \sum_{\substack{d\geq 1\\ \ell=\ell_1\ell_2}} x_{d\ell_1}\overline{x_{d\ell_2}}.
\end{equation}
By \eqref{eqn: pretrace formula}, it follows that
\begin{equation}\label{eqn: APTF}
  \begin{split}
    & \quad \sum_{j\geq0}h(t_j)A_j |u_j(z)|^2
          + \frac{1}{4\pi} \int_{-\infty}^{\infty}  h(r) A_{ir} |E(z,1/2+ir)|^2 dr \\
    & = \sum_{\ell} \frac{|y_\ell|}{\sqrt{\ell}} \sum_{\g\in \cM(\ell)} |k(\g z,z)|.
  \end{split}
\end{equation}
This identity is what we call the \emph{amplified pre-trace formula}.

To obtain upper bounds we use a test function $h(r)$
which is localized for $r$ near $T$, with $T\geq 2$ being a parameter.
We need a suitable point-pair kernel and the coming estimate.
\begin{lemma}\label{lemma: k_T}
  For all $T\geq1$, there is a point-pair kernel $k_T\in C_c^\infty([0,\infty))$, supported on
  $[0,1]$, which satisfies the following properties:
  \begin{itemize}
    \item [(i)]   The spherical transform $h_T(r)$ is positive for all $r\in \bbR\cup i\bbR$,
    \item [(ii)]  For all $T\leq r\leq T+1$, $h_T(r)\gg 1$,
    \item [(iii)] For all $u\geq 0$, $|k_T(u)|\leq T$,
    \item [(iv)]  For all $T^{-2}\leq u\leq 1$, $|k_T(u)|\leq \frac{T^{1/2}}{u^{1/4}}$.
  \end{itemize}
\end{lemma}
\begin{proof}
  See Templier \cite[Lemma 2.1]{templier2015hybrid}.
\end{proof}

\begin{lemma}\label{lemma: int}
  Let $k_T$ be as in the above lemma. Let $M:[0,1]\rightarrow\mathbb{R}_+$ be a non-decreasing
  function with finitely many discontinuities such that $M(\delta)\ll \delta^\alpha$ for some
  $\alpha>0$. Then the following holds with $\beta:=\max(1/2,1-2\alpha)$:
  \begin{align*}
    \int_0^1 |k_T(\delta)| dM(\delta)
    \ll T^{\beta}.
  \end{align*}
\end{lemma}
\begin{proof}
  See Templier \cite[Lemma 2.3]{templier2015hybrid}.
\end{proof}

%

Hence, by \eqref{eqn: APTF}, we have
\begin{equation}\label{eqn: Iwaniec-Sarnak}
  \int_{T}^{T+1}  A_{ir} |E(z,1/2+ir)|^2 dr
  \ll \sum_{\ell} \frac{|y_\ell|}{\sqrt{\ell}} \sum_{\g\in \cM(\ell)} |k_T(\g z,z)|.
\end{equation}
And now, we have three main problems to overcome to obtain a bound for $E(z,1/2+it)$.
The first problem is to relate a pointwise bound on $E(z,1/2+it)$ to an integral bound
of the type occurring in \eqref{eqn: Iwaniec-Sarnak}.
This has done by Young \cite{young2015note}, see Lemma \ref{lemma: Bound via Integral} below.
The second problem is to choose $x_n$ in order to make
the amplifier to be large on an integral of $t$'s of length
$T^{-\ve}$, but not simply be large at a single value of $t$.
The last problem is to give a good upper bound for the the right-hand
side of \eqref{eqn: Iwaniec-Sarnak}, that is, to count
lattice points efficiently.


\subsection{Summary of Young's results} \label{sec: Young}

We need the following modified version of Young \cite{young2015note}.
\begin{lemma}\label{lemma: Bound via FE}
  For $t\geq1$, and $y\gg1$, we have
  \begin{align*}
    F(z,1/2+it)\ll (t/y)^{1/2}\log^2 t+t^{1/6+\ve}.
  \end{align*}
\end{lemma}
\begin{proof}
  The proof is similar to \cite[Lemma 3.1]{young2015note}.
  So a sketch proof is enough.
  By Stirling's formula,
  \begin{align*}
  F(z,1/2+it) \ll \f{\sqrt{y}}{|\zeta(1+2it)|} \sum_{n=1}^{\infty} d(n) |K_{it}(2\pi ny)|\cosh(\pi t/2).
  \end{align*}
  We have uniform bounds on the $K$-Bessel function which we extract from the uniform asymptotic expansions due to Balogh \cite{balogh1967asymptotic}:
  \begin{equation*}
    e^{\frac{\pi}{2}t}K_{it}(u) \ll \left\{
    \begin{array}{ll}
      t^{-1/4}(t-u)^{-1/4}, & \textrm{if } 0<u<t-Ct^{1/3},\\
      t^{-1/3}, & \textrm{if } |u-t|\leq Ct^{1/3},\\
      u^{-1/4}(u-t)^{-1/4} \exp\left(-c\left(\frac{u}{t}\right)^{3/2}\left(\frac{u-t}{t^{1/3}}\right)^{3/2}\right), & \textrm{if } u>t+Ct^{1/3},\\
    \end{array}
    \right.
  \end{equation*}
  where $c$ and $C$ are fixed constants.
  With the help of this, we can assume $y\ll t$,
  since the exponential decay of the $K$-Bessel function.
  The deduction of \cite[Lemma 3.1]{young2015note} gives us
  \begin{equation*}
    F(z,1/2+it)\ll (t/y)^{1/2}\log^2 t+y^{1/2}t^{-1/3+\ve}\ll (t/y)^{1/2}\log^2 t + t^{1/6+\ve}.
    \qedhere
  \end{equation*}
\end{proof}

We also need the following Lemma of Young \cite{young2015note}.
\begin{lemma}\label{lemma: Bound via Integral}
  Suppose $y,T\gg 1$. Then
  \begin{align*}
    |E(z,1/2+iT)|^2\ll y\log^6T+\log^5T\int_{|r|\leq 4\log T}|E(z,1/2+iT+ir)|^2 dr.
  \end{align*}
\end{lemma}
\begin{proof}
  See Young \cite[Corollary 4.2]{young2015note}.
\end{proof}


\subsection{Amplifier} \label{sec: Amplifier}
It is convenient to give our amplifier now. We should modify the construction of Young \cite{young2015note}.
Let
\[
  \cP = \{ p \ {\rm prime}: N\leq p\leq 2N\}.
\]
be a large set of primes, and define
\begin{equation}\label{eqn: x_n}
  x_n := \left\{\begin{array}{ll}
                  w(n/N)\log(n)\eta_{it}(n), & {\rm if}\ n\in \cP, \\
                  0, &   {\rm otherwise},
                \end{array}\right.
\end{equation}
where $w$ is a fixed, compactly-support positive function on the positive reals, with
\begin{equation}\label{eqn: conditions of w}
  \supp(w)\subset[1,2],\quad 0\leq w(r)\leq 1, \quad
  {\rm and} \quad \int_{-\infty}^{\infty}w(r)dr \neq 0.
\end{equation}
Hence $y_\ell$ defined in \eqref{eqn: y_ell} satisfies:
\begin{equation}\label{eqn: upper bound for y_ell}
  y_\ell \ll \left\{\begin{array}{ll}
                      N, & \textrm{if } \ell=1,\\
                      \log^2 N, & \textrm{if } \ell=\ell_1\ell_2 \ \textrm{with}\ \ell_1,\ell_2\in \cP, \\
                      0, & \textrm{otherwise}.
                    \end{array}\right.
\end{equation}
The advantage of this amplifier can be seen in
Sections \ref{sec: lower bound for amplifier} and \ref{sec: counting}.




\section{A lower bound for the amplifier} \label{sec: lower bound for amplifier}

Define
\begin{equation}\label{eqn: A_N}
  A_N(t,r) = \sum_{p=2}^{\infty} w(p/N)\log(p)\eta_{it}(p)\eta_{ir}(p),
\end{equation}
where $\eta_{it}(n)$ is defined in \eqref{eqn: eta} and
$w$ is a fixed function satisfied to \eqref{eqn: conditions of w}.
Let
\begin{equation}\label{eqn: L(s)}
  L(s) = \frac{\zeta(s+it+ir)\zeta(s-it+ir)\zeta(s+it-ir)\zeta(s-it-ir)}{\zeta(2s)},
\end{equation}
for $\Re(s)>1$, where $\zeta(s)$ is the Riemann zeta-function.
Using a well-known identity of Ramanujan \cite[equation (15)]{ramanujan2000some} (see also \cite[equation (1.28)]{iwaniec2004analytic}), we derive
\[
  L(s) = \sum_{n=1}^\infty \frac{\eta_{it}(n)\eta_{ir}(n)}{n^s}.
\]
Then, by the Euler product of $L(s)$, we have
\begin{equation} \label{eqn: L'/L}
  \begin{split}
     - \frac{L'}{L}(s) & = \sum_{p}\log p \left(\f{p^{-s-it-ir}}{1-p^{-s-it-ir}}+\f{p^{-s+it-ir}}{1-p^{-s+it-ir}} \right. \\
                       &   \qquad \qquad \qquad \left. +\f{p^{-s-it+ir}}{1-p^{-s-it+ir}} + \f{p^{-s+it+ir}}{1-p^{-s+it+ir}}-\f{2p^{-2s}}{1-p^{-2s}} \right) \\
                       & = \sum_{n=1}^\infty \frac{b(n)}{n^s}, \quad \Re s>1,
  \end{split}
\end{equation}
say.
By the Taylor expansion, we derive that
\begin{equation}\label{eqn: b(n)}
  b(n) = \left\{\begin{array}{ll}
    \log(p)\eta_{it}(p)\eta_{ir}(p), & {\rm if}\ n=p,\\
    \log(p)b_{p,k}, & {\rm if}\ n=p^k,\ k\geq 2, \\
    0, & {\rm otherwise},
  \end{array}\right.
\end{equation}
with $|b_{p,k}|\leq 6$ for all $p$ prime and $k\geq 2$.
Define
\begin{equation}\label{eqn: B_N}
  B_N(t,r) = \sum_{n=1}^{\infty} w(n/N)b(n).
\end{equation}
By \eqref{eqn: conditions of w} and \eqref{eqn: b(n)}, we have
\begin{equation} \label{eqn: A_N vs B_N}
  A_N(t,r) = B_N(t,r)+O(\sqrt{N}).
\end{equation}
So we can estimate $A_N(t,r)$ by the estimation of $B_N(t,r)$.

\begin{lemma}\label{lemma: bound for amplier}
  Suppose that $\log N \gg (\log T)^{2/3+\d}$, and $t,r = T+O((\log N)^{-1-\d})$, for some fixed $\d>0$.
  Then
  \begin{equation}
    A_N(t,r) = 2N\widetilde{w}(1)(1+o(1)),
  \end{equation}
  where
  $
  \widetilde{w}(s)=\int_0^\infty w(y)y^{s-1}dy
  $
  is the Mellin transform of $w$.
\end{lemma}

Fouvry, Kowalski, and Michel \cite[Lemma 2.4]{fouvry2014algebraic} prove a result with a similar conclusion,
but their method requires $N\gg T^3$. Young \cite[Lemma 5.1]{young2015note} also gives a similar result,
but his summation is over all integers. Our proof is analogous to the proof of \cite[Lemma 5.1]{young2015note}.

\begin{proof}
  Taking a Mellin transform and by \eqref{eqn: L(s)}and \eqref{eqn: L'/L}, we derive
  \[\begin{split}
    B_N(t,r) = & \frac{1}{2\pi i} \int_{(2)} N^s \widetilde{w}(s) \left( -\frac{L'}{L}(s) \right)ds\\
    = & \frac{1}{2\pi i} \int_{(2)} N^s \widetilde{w}(s)  \left( -\frac{\zeta'}{\zeta}(s+it+ir)-\frac{\zeta'}{\zeta}(s-it+ir)\right. \\
    & \qquad \qquad \qquad \quad\ \  \left. -\frac{\zeta'}{\zeta}(s+it-ir)-\frac{\zeta'}{\zeta}(s-it-ir)+\frac{\zeta'}{\zeta}(2s) \right)ds.
  \end{split}
  \]
  Next we move the contour to the left, to one along the straight line segments $L_1,L_2,L_3$ defined by
  $L_1=\left\{ 1-\frac{c}{(\log T)^{2/3+\d/2}}+i\tau: |\tau|\leq 100T \right\}$,
  $L_2=\left\{ 1+i\tau: |\tau|\geq 100T \right\}$, and the short horizontal segments
  $L_3=\left\{ \s\pm 100iT: 1-\frac{c}{(\log T)^{2/3+\d/2}}\leq \s \leq 1 \right\}$,
  where $c$ is a small positive number such that $L(s)$ is zero-free
  on the boundary and right side of $L_1\cup L_2\cup L_3$.
  By \cite[Theorem 8.29]{iwaniec2004analytic}, we may also assume
  the Vinogradov--Korobov bound $\zeta'(s)/\zeta(s)\ll (\log |\tau|)^{2/3}(\log\log \tau)^{1/3}$
  in this region.
  The integrals along the line segments $L_2$ and $L_3$ are trivially bounded by $O(T^{-100})$ by the rapid decay of $\widetilde{w}$.
  The new line $L_1$ gives an amount that is certainly
  \begin{equation*}
    \ll N (\log T) \exp\left( -\frac{\log N}{(\log T)^{2/3+2\d/3}} \right) \ll \frac{N}{(\log T)^{100}}.
  \end{equation*}

  Now we need to analyze the residue of the poles.
  The contribution from $s=1+it+ir$ and $s=1-it-ir$ is negligible
  because of the rapid decay of $\widetilde{w}(s)$.
  The residue at $s=1+it-ir$ contributes
  \begin{equation}
    R = N^{1+it-ir} \widetilde{w}(1+it-ir).
  \end{equation}
   Write $t=r+\eta$ (by assumption, $\eta=O((\log N)^{-1-\d})$), and using Taylor expansions, we have
  \[
    R = N^{1+i\eta} \widetilde{w}(1+i\eta) = N\widetilde{w}(1)(1 + O(|\eta|\log N)) = N\widetilde{w}(1)(1 + o(1)).
  \]
  By a similar argument for the residue at $s=1-it+ir$, we have
  \begin{equation*}
    B_N(t,r) = 2N\widetilde{w}(1)(1+o(1)).
  \end{equation*}
  Hence, by \eqref{eqn: A_N vs B_N}, we prove the lemma.
\end{proof}

\section{Counting lattice points} \label{sec: counting}

We need the counting results in Templier \cite{templier2015hybrid}, where
he counted lattice points respect to the level and the parameter.
Here, we restrict the level to be $1$.

For $z\in\bbH$, $\d>0$ and the integer $\ell$, let $\cM(z,\ell,\d)$ be the finite set of matrices $\g=\left(\begin{smallmatrix}a&b\\c&d\end{smallmatrix}\right)$ in $M_2(\bbZ)$ such that
\begin{equation}
  \det(\g)=\ell, \quad u(\g z,z)\leq \d.
\end{equation}
Denote by $M = M(z,\ell,\d)$ its cardinality.
We split the counting $M$ of matrices $\g=\left(\begin{smallmatrix}a&b\\c&d\end{smallmatrix}\right)$ as
\[
  M = M_* + M_u + M_p
\]
according to whether $c\neq0$ and $(a+d)^2\neq4\ell$ (generic), or $c=0$ and $a\neq d$ (upper-triangular),
or $(a+d)^2=4\ell$ (parabolic).

Now we should recall some results on counting lattice points in Templier \cite{templier2015hybrid}.
\begin{lemma}\label{lemma: M_*}
  For any $z=x+iy\in D$, and any integer $L$ and $0<\d<1$, we have
  \begin{equation}
    \sum_{1\leq \ell\leq L}M_*(z,\ell,\d) \ll L^\ve \le( \frac{L}{y} + L^{3/2}\d^{1/2} + L^2\d \ri).
  \end{equation}
\end{lemma}

\begin{proof}
  See \cite[Lemma 1.3]{templier2015hybrid}.
\end{proof}

Recall that $\cF(q)$ is the set of $z\in\bbH$ such that $\Im z\geq \Im Az$ for all $A\in A_0(q)$.
The original results of the following two lemmas in
Templier \cite{templier2015hybrid} were considered in the region $\cF(q)$.
In our case, the corresponding region $\cF(1)$ becomes the well-known fundamental domain of $SL_2(\mathbb{Z})$,
which is $D=\{z=x+iy\in\bbH:|x|\leq 1/2,|z|\geq 1\}$.

\begin{lemma}\label{lemma: M_u}
  For any $z=x+iy\in D$, and any integer $L$ and $0<\d<1$, the following estimate holds where $\ell_1,\ell_2$ run over primes:
  \begin{equation}
    \sum_{1\leq \ell_1,\ell_2\leq L} M_u(z,\ell_1\ell_2,\d) \ll L^\ve \le( L + L^3\d^{1/2}y \ri).
  \end{equation}
\end{lemma}

\begin{proof}
  See \cite[Lemma 4.3]{templier2015hybrid}.
\end{proof}

\begin{remark}
  The original proof of Lemma \ref{lemma: M_u} comes from \cite[Lemma 4]{harcos2013sup}, where three results were stated
  according to $\ell=\ell_1\ell_2$, $\ell=\ell_1\ell_2^2$, or $\ell=\ell_1^2\ell_2^2$.
  We remark that there is an extra condition which states that
  $L\leq q^{O(1)}$, where $q$ is the level. However, in our case ($\ell=\ell_1\ell_2$), this condition is not needed in the proof.
  So we remove it here.
\end{remark}

\begin{lemma}\label{lemma: M_p}
  For any $z=x+iy\in D$, we have
  \begin{equation}
    M_p(z,\ell,\d) \ll (1+\ell^{1/2}\d^{1/2}y+\ell^{3/4}\d^{3/8}y^{-1/2}) \d_\square(\ell),
  \end{equation}
  where $\d_\square(\ell) = 1,0$ depending on whether $\ell$ is a perfect square or not.
\end{lemma}

\begin{remark} \label{rmk: BHMM}
  A better bound could be found in \cite[Lemma 4.4]{templier2015hybrid}.
  However, in the recent preprint \cite[Remark 5]{blomer2016sup-norm}, the authors explained
  that it remained unclear to them to get this better bound. Instead, they provides
  an alternative bound of a number field version. Here, for completeness, we use
  their technique of \cite[Lemma 13]{blomer2016sup-norm} to give this result which is weaker
  than \cite[Lemma 4.4]{templier2015hybrid} but sufficient for us.
\end{remark}

\begin{proof}
  Note that $\ell$ is always a perfect square.
  Firstly, we consider the case $c=0$.
  By $(a+d)^2=4\ell=4ad$, we have $a=d=\pm \ell^{1/2}$.
  So
  \begin{equation*}
    \delta\geq u(\g z,z)=\f{|az+b-cz^2-dz|^2}{\ell y^2}=\frac{b^2}{\ell y^2}.
  \end{equation*}
  Thus, we have $\#(b)\ll 1+\ell^{1/2}\d^{1/2}y$. And then, we deduce that the number of possibilities for the parabolic matrix
  $\g=\left(\begin{smallmatrix}a&b\\c&d\end{smallmatrix}\right)$
  with $c=0$ is
  \begin{align}\label{c=0}
    \ll (1+\ell^{1/2}\d^{1/2}y) \d_\square(\ell).
  \end{align}

  Now we count the number of matrices
  $\g=\left(\begin{smallmatrix}a&b\\c&d\end{smallmatrix}\right)\in M_2(\mathbb{Z})$
  with $c\neq0$ such that
  \begin{equation*}
    u(\g z,z)=\f{|\g z-z|^2}{\Im (\g z)\Im(z)}\leq \delta,
  \end{equation*}
  and $(a+d)^2=4\ell=4(ad-bc)$.
  We have
  \begin{align*}
    \delta\geq u(\g z,z)\geq \f{(\Im (\g z-z))^2}{\Im (\g z)\Im(z)}=\left(\f{\ell-|cz+d|^2}{\ell^{1/2}|cz+d|}\right)^2
    =\left(\f{\ell^{1/2}}{|cz+d|}-\f{|cz+d|}{\ell^{1/2}}\right)^2,
  \end{align*}
  which implies that
  \begin{align}\label{cz+d}
    |cz+d|=\ell^{1/2}(1+O(\delta^{1/2})).
  \end{align}
  On the other hand, we can also get
  \begin{equation*}
    \delta\geq u(\g z,z)=\f{|az+b-cz^2-dz|^2}{\ell y^2}\geq \f{(\Im (az+b-cz^2-dz))^2}{\ell y^2}
    =\f{(2cx-a+d)^2}{\ell}.
  \end{equation*}
  This deduces that
  \begin{align}\label{a-d}
    \f{|2cx-a+d|}{\ell^{1/2}}\leq \delta^{1/2}.
  \end{align}
  With the help of this and $|a+d|=2\ell^{1/2}$, we obtain
  \begin{align*}
  2||cx+d|-\ell^{1/2}|=||2cx+2d|-|a+d||\leq |2cx-a+d|\leq \ell^{1/2}\delta^{1/2},
  \end{align*}
  which gives
  \begin{align}\label{cx+d}
    |cx+d|=\ell^{1/2}(1+O(\delta^{1/2})).
  \end{align}
  We can estimate $c$ by \eqref{cz+d} and \eqref{cx+d} as follows:
  \begin{align*}
     c^2y^2=|cz+d|^2-|cx+d|^2=\ell(1+O(\delta^{1/2}))^2-\ell(1+O(\delta^{1/2}))^2\ll \ell\delta^{1/2},
  \end{align*}
  and hence
  \begin{align}\label{estimate c}
     c\ll \f{\ell^{1/2}\d^{1/4}}{y}.
  \end{align}
  Combining the above together with \eqref{a-d} and the identity $(a-d)^2+4bc=0$, we show that
  (for any fixed $c\neq0$, we know $a-d$ is divisible by a fixed number at least $c^{1/2}$)
  \begin{align*}
  \sharp(c,a-d)\ll\sum_{c\ll \ell^{1/2}\d^{1/4}y^{-1}}\f{c+\ell^{1/2}\d^{1/2}}{c^{1/2}}
  \ll \ell^{3/4}\d^{3/8}y^{-1/2}.
  \end{align*}
  Finally, since the trace $a+d=\pm2\ell^{1/2}$, the number
  of possibilities for the parabolic matrix
  $\g=\left(\begin{smallmatrix}a&b\\c&d\end{smallmatrix}\right)$
  with $c\neq0$ is
  \begin{align}\label{c neq 0}
  \sharp(c,a-d,a+d)\ll \ell^{3/4}\d^{3/8}y^{-1/2}.
  \end{align}
  We complete this proof by \eqref{c=0} and \eqref{c neq 0}.
\end{proof}


\section{Proof of Theorem \ref{thm: level 1}} \label{sec: thm1}

We first suppose $y\gg1$.
By Lemma \ref{lemma: Bound via Integral}, we have that
\begin{equation}\label{eqn: E to int}
    |E(z,1/2+iT)|^2\ll y\log^6T+\log^5T\int_{|r|\leq 4\log T}|E(z,1/2+iT+ir)|^2 dr.
\end{equation}
On the right hand side above, we dissect the integral into subintervals, each of length $\asymp (\log T)^{-2}$, say.
Let $U$ be one of these intervals, and let $t_U$ be the left endpoint of $U$.
Then by Lemma \ref{lemma: bound for amplier}, we have
\begin{equation}\label{eqn: int with amplier}
  \begin{split}
    & \quad \int_{r\in U} |E(z,1/2+iT+ir)|^2 dr \\
    & \ll N^{-2} \int_{r\in U} |A_N(T+t_U,T+r)|^2 |E(z,1/2+iT+ir)|^2 dr.
  \end{split}
\end{equation}
By \eqref{eqn: Iwaniec-Sarnak} and \eqref{eqn: A_N}, where $x_n$ is defined as in \eqref{eqn: x_n} with $t=T+t_U$, we have
\begin{equation}\label{eqn: int to counting}
  \int_{r\in U} |A_N(T+t_U,T+r)|^2 |E(z,1/2+iT+ir)|^2 dr
  \ll \sum_{\ell} \frac{|y_\ell|}{\sqrt{\ell}} K_{T+t_U}(z,\ell),
\end{equation}
where $y_\ell$ is as in \eqref{eqn: y_ell} and
\begin{equation}
    K_T(z,\ell) := \sum_{\g\in \cM(\ell)} |k_T(\g z,z)| 
         = \int_{0}^1 |k_T(\d)| d M(z,\ell,\d).
\end{equation}
Hence we first need to estimate the quantity
\begin{equation}
  A(z,\d) := \sum_{\ell} \frac{|y_\ell|}{\sqrt{\ell}} M(z,\ell,\d).
\end{equation}
Since $M = M_* + M_u + M_p$, we decompose $A = A_* + A_u + A_p$ accordingly.
By \eqref{eqn: upper bound for y_ell}, we have
\begin{equation}
  A(z,\d) \ll NM(z,1,\d) + \frac{\log^2 N}{N}\underset{\ell_1,\ell_2\in\cP}{\sum\sum} M(z,\ell_1\ell_2,\d).
\end{equation}
Hence together with the fact that $M_u(z,1,\d)=0$ and
Lemmas \ref{lemma: M_*}, \ref{lemma: M_u} and \ref{lemma: M_p},
we get the following proposition.
\begin{proposition}\label{prop: A}
  Let $z=x+iy\in D$, $0<\d<1$. Then
  \begin{equation*}
    \begin{split}
      A_*(z,\d) & \ll (NT)^\ve \le( N + N^2\d^{1/2} + N^3\d \ri), \\
      A_u(z,\d) & \ll (NT)^\ve \le( 1 + N^2\d^{1/2}y \ri), \\
      A_p(z,\d) & \ll (NT)^\ve \le( N + N\d^{1/2}y + N^{3/2}\d^{3/8}y^{-1/2} \ri).
    \end{split}
  \end{equation*}
\end{proposition}

We now execute the integration over $\d$ in the Stieltjes integral
\begin{equation*}
  \sum_{\ell} \frac{|y_\ell|}{\sqrt{\ell}} K_T(z,\ell) = \int_{0}^{1} |k_T(\d)| d A(z,\d).
\end{equation*}
From Lemma \ref{lemma: int} we see that we can make the substitutions
$(\d\rightsquigarrow T^{1/2})$, $(\d^{1/2}\rightsquigarrow T^{1/2})$, $(\d^{3/8}\rightsquigarrow T^{1/2})$ and $(1\rightsquigarrow T)$ starting
from the upper-bound for $A(z,\d)$ to obtain the bound for the Stieltjes integral as \cite[\S 6]{templier2015hybrid} did.
Here we emphasize that this formula is used with $T+t_U$ in place of $T$, and then one needs to sum up the
resulting bounds over the various subintervals $U\subset[-4\log T, 4\log T]$.
Altogether we obtain from Proposition \ref{prop: A} and after some simplifications that:
\begin{equation}\label{eqn: upper bound for sum over ell}
  \begin{split}
    & \sum_{\ell} \frac{|y_\ell|}{\sqrt{\ell}} K_{T+t_U}(z,\ell) \\
      \ll & (N(T+t_U))^\ve \le( N(T+t_U) + N^3 (T+t_U)^{1/2} + N^2(T+t_U)^{1/2}y \ri)\\
    \ll & (NT)^\ve \le( NT + N^3 T^{1/2} + N^2T^{1/2}y \ri).
  \end{split}
\end{equation}

From the bound via Fourier expansion in Lemma \ref{lemma: Bound via FE}, we can assume without loss of
generality when establishing Theorem \ref{thm: level 1} that
\begin{equation*}
  y\ll T^{1/4}.
\end{equation*}
Combining \eqref{eqn: E to int},  \eqref{eqn: int with amplier}, \eqref{eqn: int to counting},
with \eqref{eqn: upper bound for sum over ell}, we obtain that
\begin{equation}\label{eqn: E^2<<1}
  |E(z,1/2+iT)|^2 \ll (NT)^\ve \left( TN^{-1} + N T^{1/2} + T^{1/2}y \right)
\end{equation}
We choose
\begin{equation}
  N := T^{1/4}.
\end{equation}
This yields
\begin{equation*}
  |E(z,1/2+iT)| \ll T^{3/8+\ve},
\end{equation*}
which implies that
\begin{align}\label{Thm1 y>>1}
E(z,1/2+iT) = y^{1/2+iT}+\varphi(1/2+iT)y^{1/2-iT} + O(T^{3/8+\ve}),
\end{align}
provided $z\in \bigcup_{k\in\mathbb{Z}}(k+D)$.

Now, we consider the case $z\notin \bigcup_{k\in\mathbb{Z}}(k+D)$. By the definition of fundamental domain,
there exists $\gamma=\left(\begin{smallmatrix} a&b\\ c&d \end{smallmatrix}\right)\in SL_2(\mathbb{Z})$,
such that $\gamma z\in D$. We see that $c\neq0$, since $\gamma z\notin D$ if $c=0$.
So we have
$$
\Im \gamma z\ll \frac{y}{c^2y^2}\ll y^{-1}.
$$
With the help of \eqref{Thm1 y>>1}, we get
\begin{align}\label{Thm1 y<<1}
E(z,1/2+iT)=E(\gamma z,1/2+iT)\ll y^{-1/2}+T^{3/8+\ve}.
\end{align}
Based on the above arguments, for any $z\in\bbH$, we obtain
\begin{equation}\label{Thm1}
    E(z,1/2+iT) = y^{1/2+iT}+\varphi(1/2+iT)y^{1/2-iT} + O(y^{-1/2}+T^{3/8+\ve}),
  \end{equation}
as claimed in Theorem \ref{thm: level 1}.

\begin{remark} \label{rmk: Young}
  In Young \cite[Equation (6.3)]{young2015note}, for $y,T\gg1$, he proved that
  \begin{equation}
    |E(z,1/2+iT)|^2 \ll (NT)^\ve \le( TN^{-1} + T^{1/2}(N + N^{1/2}y) \ri).
  \end{equation}
  Compared with \eqref{eqn: E^2<<1}, one can see that the term $T^{1/2}N^{1/2}y$ may be improved to be
  $T^{1/2}y$ by our amplifier and counting method.
  This is the reason why we can improve Young's result slightly.
\end{remark}


\section{Proof of Theorem \ref{thm: level q}} \label{sec: thm2}

Let
  \begin{equation}\label{eqn: F_a}
    F_\fa(z,s)=E_\fa(z,s)-\delta_\fa y^s -\varphi_\fa(s)y^{1-s}.
  \end{equation}
Then, by \eqref{eqn: FE}, \eqref{eqn: FE of E_a}, and \eqref{eqn: E_a to E}, we deduce that
\begin{equation}\label{eqn: F_a to F}
  F_\fa(z,s) = \zeta_q(2s)\mu(v)(qv)^{-s}
     \sum_{\beta|v}\sum_{\gamma|w}\mu(\beta\gamma)\beta^s\gamma^{-s}F(\beta\gamma z,s),
\end{equation}
where $\fa\sim1/v$, $v|q$, and $w=q/v$.
Since $q$ is square-free, for $\Re(s)=1$, we have
\begin{equation}\label{eqn: zeta_q<<}
  \begin{split}
    |\zeta_q(s)| & \leq \prod_{p|q}(1-p^{-1})^{-1} =q\varphi(q)^{-1}\ll \log\log q,
  \end{split}
  \end{equation}
  where $\varphi(q)$ is the Euler function.
By Theorem \ref{thm: level 1}, we have
\begin{equation}\label{eqn: F<<}
  F(z,1/2+iT) \ll y^{-1/2}+T^{3/8+\ve}.
\end{equation}
Hence, we obtain
\begin{equation}\label{eqn: F_a<}
  \begin{split}
     F_\fa(z,1/2+iT) & \ll \frac{\log q}{(qv)^{1/2}}\sum_{\beta|v}\sum_{\gamma|w}\beta^{1/2}\gamma^{-1/2}|F(\beta\gamma z,1/2+iT)| \\
       & \ll \frac{\log q}{(qv)^{1/2}}\sum_{\beta|v}\sum_{\gamma|w}\beta^{1/2}\gamma^{-1/2}((\beta\gamma y)^{-1/2}+T^{3/8+\ve}) \\
       & \ll (qv)^{-1/2+\ve}y^{-1/2}+q^{-1/2+\ve}T^{3/8+\ve}.
  \end{split}
\end{equation}
So, we prove \eqref{eqn: E_a<<}.
In particular, if $\fa\sim\infty$, then since $1/q\sim \infty$ as cusps, we have $v=q$.
And, by \eqref{eqn: F_a<}, for $y\gg1/q$, we obtain
\begin{equation}
  F_\fa(z,1/2+iT) \ll q^{-1/2+\ve}T^{3/8+\ve}.
\end{equation}
Now, Theorem \ref{thm: level q} is proved.

\noindent{\bf Acknowledgment}.
The authors would like to thank Professor Jianya Liu for his constant
encouragement. They would also like to thank the referee for very
useful comments.


\medskip
\bibliographystyle{plain}
\bibliography{bib_sup-norm-revised4}

\end{document}